\documentclass[12pt,a4paper]{article}

\usepackage{amsmath}
\usepackage{amssymb}
\usepackage{amsfonts}
\usepackage[{left=1.5cm,right=1.5cm, top=2cm, bottom=3cm}]{geometry}

\newtheorem{theorem}{Theorem}[section]

\newtheorem{corollary}[theorem]{Corollary}

\newtheorem{lemma}[theorem]{Lemma}

\newenvironment{proof}[1][Proof]{\noindent\textbf{#1.} }{\ \rule{0.5em}{0.5em}}

\newcommand{\Pol}{\mathbb{F}\left[\alpha_1, \alpha_2\right]}
\newcommand{\F}{\mathbb{F}}
\newcommand{\Q}{\mathbb{Q}}
\newcommand{\R}{\mathbb{R}}
\newcommand{\N}{\mathbb{N}}
\newcommand{\A}{\mathbb{A}}
\newcommand{\SO}{SO(2,\mathbb{F})}
\newcommand{\GL}{GL(2,\mathbb{F})}

\begin{document}

\title{An Algebraic Interpretation of the Super Catalan Numbers}
\author{Kevin Limanta\\School of Mathematics and Statistics, UNSW Sydney\\\texttt{k.limanta@unsw.edu.au}
}
\date{}
\maketitle

\begin{abstract}
    We extend the notion of polynomial integration over an arbitrary circle $C$ in the Euclidean geometry over general fields $\F$ of characteristic zero as a normalized $\F$-linear functional on $\Pol$ that takes polynomials that evaluate to zero on $C$ to zero and is $\SO$-invariant. This allows us to not only build a purely algebraic integration theory in an elementary way, but also give the super Catalan numbers
    \begin{align*}
        S(m,n) = \frac{(2m)!(2n)!}{m!n!(m+n)!}
    \end{align*}
    an algebraic interpretation in terms of values of this algebraic integral over some circle applied to the monomials $\alpha_1^{2m}\alpha_2^{2n}$.
\end{abstract}

\section{Introduction}

This is the second of our series of papers building an integration theory of polynomials over unit circles over a general field $\F$. The first paper \cite{Limanta} deals with the case $\F$ is finite of odd characteristic, in which the family of integers $S(m,n)$ called the \textit{super Catalan numbers} and their closely related family of rational numbers $\Omega(m,n)$ called the \textit{circular super Catalan numbers} play a prominent role. They are defined as
\begin{align*}
S(m,n) := \frac{(2m)!(2n)!}{m!n!(m+n)!}, \quad \Omega(m,n) := \frac{S(m,n)}{4^{m+n}} 
\end{align*}
and are indexed by two elements in $\N$ which for us includes $0$.

The super Catalan numbers were first introduced by Catalan \cite{Catalan} in 1874  and the first modern study of these numbers was initiated by Gessel \cite{Gessel} in 1992. They generalized the Catalan numbers $c_n$ since $S(1,n) = 2c_n$. The integrality of $S(m,n)$ can be observed from the relation $4S\left(m,n\right) = S\left(m+1,n\right) + S\left( m,n+1\right)$ which yields the Pascal-like property $\Omega(m,n) = \Omega(m+1,n) + \Omega(m,n+1)$.

No combinatorial interpretation of $S(m,n)$ is known for general $m$ and $n$ to date, in contrast to over $200$ interpretations of the Catalan numbers \cite{Stanley}. However, for $m = 2$, there is some interpretations in terms of cubic trees by Pippenger \cite{Pippenger} and blossom trees by Schaeffer \cite{Schaeffer}, and when $m = 2, 3$, as pairs of Dyck paths with restricted heights by Gessel and Xin \cite{Gessel2}. When $n = m+s$ for $0\leq s\leq3$, Chen and Wang showed that there is an interpretation in terms of restricted lattice paths \cite{Chen}. There is also some weighted interpretation of $S\left(m,n\right)$ as a certain value of Krawtchouk polynomials by the work of Georgiadis, Munemasa, and Tanaka \cite{Georgiadis} and another in terms of positive and negative $2$-Motzkin paths by Allen and Gheorghiciuc \cite{Allen}.

The aim of this paper is twofold. The first one is to build, in a rather elementary way, a polynomial integration theory over circles in the Euclidean geometry over general fields of characteristic zero without recourse to the usual Riemann integral and limiting processes. We shall see that this allows us to give the super Catalan numbers a purely algebraic interpretation, which is our second objective.

Here and throughout, $\F$ is a general field of characteristic zero with multiplicative identity $1_{\F}$ or sometimes just $1$ if the context is clear. We denote by $\Pol$ the algebra of polynomials in $\alpha_1$ and $\alpha_2$ over $\F$ with multiplicative identity $\mathbf{1}$. Our algebraic integral over a circle $C$ is a linear functional $\phi$ on $\Pol$, called a circular integral functional with respect to $C$, which satisfies three conditions: $\phi(\mathbf{1}) = 1_{\F}$ (\textit{Normalization}), $\phi(\pi) = 0$ whenever $\pi$ evaluates to the zero function on $C$ (\textit{Locality}), and $\phi$ is rotationally-invariant (\textit{Invariance}).

When $\F = \R$, there is a well-known formula for the integral of polynomials on the unit sphere $S^{n-1}$ (see \cite{Baker} or \cite{Folland}).

\begin{theorem}\label{Folland formula}
Let $n\geq 2$ and $S^{n-1}$ denote the $\left(n-1\right)$-dimensional unit sphere in $\mathbb{R}^{n}$. If $\mu $ is the usual rotationally invariant measure on $S^{n-1}$, then by writing $b_i = \frac{1}{2}\left(d_i + 1\right)$,
\begin{align*}
    \int_{S^{n-1}} x_{1}^{d_1} x_{2}^{d_2} \ldots x_n^{d_n}\, d\mu = 
    \begin{cases}
        \dfrac{2}{\Gamma\left(b_1 + b_2 + \cdots + b_n \right)}{\displaystyle\prod\limits_{i=1}^{n}} \Gamma\left(b_i\right) &\text{if each } d_i \text{ is even} \\ 
        0 &\text{otherwise.}
    \end{cases}
\end{align*}
\end{theorem}
We may see from Theorem \ref{Folland formula} above, when $n = 2$, $d_1 = 2m$, and $d_2 = 2n$, we obtain
\begin{align*}
    \frac{2\Gamma\left(b_1\right)\Gamma\left(b_2\right)}{\Gamma\left(b_1 + b_2\right)} = \frac{2\Gamma\left(m+\frac{1}{2}\right)\Gamma\left(n+\frac{1}{2}\right)}{\Gamma\left(m+n+1\right)} = 2\pi \frac{(2m)!(2n)!}{4^{m+n}m!n!(m+n)!} = 2\pi \Omega(m,n)
\end{align*}
so if the integral is normalized, we get just the circular super Catalan numbers.

In \cite{Limanta}, we showed that the polynomial integration theory over finite field of odd characteristics is analogous to the $\F = \R$ case, which we summarize below.
\begin{theorem}\label{Thm formula blue CIF SC} Let $p>2$ be a prime and $q = p^r$ for some $r \in \N$. In the Euclidean geometry over $\mathbb{F}_{q}$ with multiplicative identity $1_q$, the unit circle is $S^1 = \left\{\left[x_1,x_2\right] \in \F_q^2 \colon x^2 + y^2 = 1_q\right\}$. Let $k$ and $l$ be any natural numbers for which $0 \leq k+l < q-1$. Then the functional $\psi_{b,q}$ on $\F_q\left[\alpha_1, \alpha_2\right]$ given by
\[
\psi_{b,q}\left(\alpha_1^{k}\alpha_2^{l}\right) = -\left(\frac{-1}{p}\right)^r \sum_{\left[x,y\right] \in S^1} x^k y^l =
\begin{cases}
\Omega\left(m,n\right) \mathrm{mod}\, p &\mbox{if } k=2m \mbox{ and } l=2n \\
0 & \mbox{otherwise}
\end{cases}
\]
is the unique circular integral functional with respect to $S^1$. Here $\left(\frac{-1}{p}\right)$ is the usual Legendre symbol.
\end{theorem}
Now we present our main result. For $a \in \Q$, $a1_{\F}$ is the embedding of $a$ in $\F$. The unit circle $S^1$ in this setting will be defined in the next section.
\begin{theorem}\label{Existence and Uniqueness Theorem Intro}
For any $k, l \in \N$, the linear functional $\psi$ on $\Pol$ defined by
\begin{equation}
\psi\left(\alpha_1^k \alpha_2^l\right) =\left\{ 
\begin{array}{cl}
\Omega\left( m,n\right)1_{\F} & \text{if } k=2m \text{ and } l=2n\text{,} \\ 
0 & \text{otherwise}
\end{array}%
\right. \label{CIF blue geometry}
\end{equation}
is the unique circular integral functional with respect to $S^1$.
\end{theorem}



\section{Circular Integral Functional}

Denote by $\A = \A(\F)$ the two-dimensional \textit{affine plane} $\left\{\left[x,y\right] \ \colon \ x,y\in \F\right\}$, with the objects $\left[x,y\right]$ called \textit{points}. There is then the space $\F^{\A}$ consisting of functions from $\A$ to $\F$ which is an $\F$-algebra under pointwise addition and multiplication, and the evaluation map $\varepsilon \colon \Pol \rightarrow \F^{\A}$ which is an algebra homomorphism. Clearly we may regard $\F\left[\alpha_1\right]$ as a subalgebra of $\Pol$. Recall that any non-zero polynomial in $\F\left[\alpha_1\right]$ of degree $d$ has at most $d$ roots.

The group $\GL$ of invertible $2 \times 2$ matrices with entries in $\F$ left acts on $\Pol$ as follows: if $\pi = \pi\left(\alpha_1, \alpha_2\right)$ then
\begin{equation}
h \cdot \pi = \begin{pmatrix} h_{11} & h_{12} \\ h_{21} & h_{22} \end{pmatrix}\cdot \pi := \pi \left(
h_{11}\alpha_1 +h_{21}\alpha_2 ,h_{12}\alpha_1 + h_{22}\alpha_2 \right). \label{G action}
\end{equation}
Additionally, $\GL$ right acts on $\A$ and left acts on $\F^{\A}$ as follows:
\begin{align*}
    \left[x_1, x_2\right] \cdot h &:= \left[h_{11}x_1 + h_{21} x_2, h_{12}x_1 + h_{22} x_2\right] \\
    (h \cdot f)(x_1, x_2) &:= f\left(\left[x_1,x_2\right] \cdot h\right) = f(h_{11}x_1 + h_{21} x_2, h_{12}x_1 + h_{22} x_2).
\end{align*}

The group $\SO$ of  matrices $h$ that satisfy $h^{-1} = h^{T}$ of determinant $1_{\F}$ is then a subgroup of $\GL$ and is called the rotation group. The action of $\SO$ on $\Pol$ is induced as the restriction of the action of $\GL$ on $\Pol$. This action respects evaluation: for any $h \in \SO$ and $\pi \in \Pol$,
\begin{equation}
\varepsilon\left(h \cdot \pi \right) = h \cdot \varepsilon \left( \pi \right).  \label{SO action respects evaluation}
\end{equation}
In a similar manner, we also get an action of $\SO$ on $\A$ and $\F^{\A}$.

We define a symmetric bilinear form on $\A$, given by $\left[ x_{1},y_{1}\right] \cdot \left[ x_{2},y_{2}\right] := x_{1}x_{2}+y_{1}y_{2}$. The associated quadratic form $\left[x,y\right] \cdot \left[x,y\right] = x^2 + y^2$ then gives rise to the (Euclidean) unit circle
\begin{equation*}
S^1 = S^1\left(\F\right) := \left\{\left[x,y\right] \in \A \colon x^2 + y^2 = 1_{\F}\right\}.
\end{equation*}
\begin{lemma}\label{Parametrization of S^1}
    Each point on $S^1$ except $[-1,0]$ can be written as $\left[\frac{1-u^{2}}{1+u^{2}},\frac{2u}{1+u^{2}}\right]$ for some $u\in \F$ such that $1+u^2 \neq 0$. Consequently, $S^1$ is an infinite set.
\end{lemma}
\begin{proof}
    The identity $\left(\frac{1-u^{2}}{1+u^{2}}\right)^{2} + \left(\frac{2u}{1+u^{2}}\right)^{2}=1_{\F}$ holds for all $u \in \F$ for which $u^2 \neq -1$. The line $y=ux+u$ through the points $\left[-1,0\right]$ and $\left[0,u\right]$ intersects $S^1$ in exactly two points, $\left[-1,0\right]$ and $\left[\frac{1-u^{2}}{1+u^{2}}, \frac{2u}{1+u^{2}}\right]$. Hence every point on $S^1$ except $\left[-1,0\right]$ corresponds to exactly one $u \in \F$ for which $u^2 \neq -1$. Since there are infinitely many $u \in \F$ for which $u^2 \neq -1$, $S^1$ is an infinite set.
\end{proof}

\begin{corollary}\label{Parametrization of SO}
    The rotation group $\SO$ admits a parametrization
    \begin{align*}
    \SO =\left\{h_u = \frac{1}{1+u^{2}} 
    \begin{pmatrix}
        1-u^{2} & -2u \\ 
        2u & 1-u^{2}
    \end{pmatrix}
    \colon \ u \in \F, u^2 \neq -1\right\} \cup \left\{
-I\right\}.
    \end{align*}
\end{corollary}

We now introduce the central object of this paper: a linear functional on $\Pol$ that generalizes normalized integration over the Euclidean unit circle over $\R$. We say that a linear functional $\phi \colon \Pol \rightarrow \F$ is a \textit{circular integral functional} on $\Pol$ with respect to $S^1$ precisely when it satisfies the following three conditions:

\begin{description}
\item[(Normalization)] For the multiplicative identity $\mathbf{1}$ of $\Pol$, we have $\phi\left(\mathbf{1}\right) = 1_{\F}$.

\item[(Locality)] If $\pi \in \Pol$ such that $\varepsilon(\pi)$ is the zero function on $S^1$, then $\phi(\pi) = 0$.

\item[(Invariance)] The functional $\phi$ is $\SO$-invariant: $\phi(h \cdot \pi) = \phi(\pi)$ for any $\pi \in \Pol$ and $h \in \SO$.
\end{description}

\section{Existence and Uniqueness}
Our strategy to prove Theorem \ref{Existence and Uniqueness Theorem Intro} is divided into two main steps. First, we show that $\psi$ satisfies the Normalization, Locality, and Invariance conditions. Next, we demonstrate that if such a circular integral functional $\Pol$ with respect to $S^1$ exists, it is uniquely determined.

It is easy to see that the Normalization condition holds. The next two lemmas are needed to prove the Locality of $\psi$.

\begin{lemma} \label{Projection of S^1 on each axis}
Both $S^1_{x_1} = \left\{x_1 \in \mathbb{F} \, \colon \, \left[x_1,x_2\right] \in S^1\right\}$ and $S^1_{x_2} = \left\{x_2 \in \mathbb{F} \, \colon \, \left[x_1,x_2\right] \in S^1\right\}$ have infinitely many elements. 
\end{lemma}
\begin{proof}
    For any $\left[x_1,x_2\right] \in S^1$, we have that $\left[x_2,x_1\right] \in S^1$, so $S^1_{x_1} = S^1_{x_2}$. If $S^1_{x_1} = S^1_{x_2}$ is finite, then so is $S^1_{x_1} \times S^1_{x_2}$ and consequently $S^1$. This contradicts the fact that $S^1$ is infinite from Lemma \ref{Parametrization of S^1}.
\end{proof}

The crucial property of polynomials in $\Pol$ that evaluate to the zero function on $S^1$ is that they must lie in $\left\langle\alpha_1^2 + \alpha_2^2 -1\right\rangle$, the ideal generated by $\alpha_1^2 + \alpha_2^2 -1$. We offer an elementary proof below by utilising the multivariate polynomial division which requires a choice of monomial ordering. This has a flavour of Hilbert's Nullstellensatz which usually works over algebraically closed fields, although our argument does not assume that $\F$ is an algebraically closed field.
\begin{lemma}
\label{Lemma division algorithm} If $\pi \in \Pol$ satisfies $\varepsilon(\pi) = 0$ on $S^1$, then $\pi \in \left\langle\alpha_1^2+\alpha_2^2-1\right\rangle$.
\end{lemma}

\begin{proof}
Fix a monomial ordering $\preccurlyeq$ such that $\alpha_1^{k_1}\alpha_2^{l_1} \preccurlyeq \alpha_1^{k_2}\alpha_2^{l_2}$ if either $k_1 < k_2$ or $k_1 = k_2$ and $l_1 < l_2$. With respect to this ordering, any $\pi \in \Pol$ can be written as $\pi = \left(\alpha_1^2 + \alpha_2^2 - 1\right) \pi _{0} + \alpha_2 \omega + \rho$ for some $\pi _{0} \in \Pol$ and $\omega, \rho \in \mathbb{F}\left[\alpha_1\right]$.

Since $\pi$ evaluates to the zero function on $S^1$, we have that
\begin{equation}
    0 = \varepsilon(\pi)(x_1,x_2) = x_2 \varepsilon(\omega)(x_1,x_2) + \varepsilon(\rho)(x_1,x_2) = x_2 \varepsilon(\omega)(x_1,0) + \varepsilon(\rho)(x_1,0) \label{Evaluation of pi after poly division}
\end{equation}
for all $\left[x_1,x_2\right] \in S^1$, with the last equation is due to $\omega, \rho \in \F\left[\alpha_1\right]$. Now consider the set $S^1_\ast = \left\{\left[x_1,x_2\right] \in S^1 \, \colon \, x_2 \neq 0\right\}$ which is non-empty since $[0,1] \in S^1_\ast$. For any $\left[x_1,x_2\right] \in S^1_\ast$, the point $\left[x_1,-x_2\right] \in S^1_\ast$ is different from $\left[x_1,x_2\right]$ so (\ref{Evaluation of pi after poly division}) forces $\varepsilon(\omega)(x_1, 0) = \varepsilon(\rho)(x_1, 0) = 0$ for all $\left[x_1, x_2\right] \in S^1$. By Lemma \ref{Projection of S^1 on each axis}, $\omega$ and $\rho$ have infinitely many roots, so they must be the zero polynomial.
\end{proof}

\begin{theorem}[Locality of $\protect\psi$]
\label{Theorem Locality psi_b} The linear functional $\psi $ satisfies the Locality condition.
\end{theorem}

\begin{proof}
If $\pi \in \Pol$ evaluates to the zero function on $S^1$, then $\pi \in \left\langle \alpha_1^{2} + \alpha_2^{2}-1\right\rangle$ from Lemma \ref{Lemma division algorithm}. By linearity, it suffices to show that
\begin{align*}
\psi\left( \left( \alpha_1^{2} + \alpha_2^{2} - 1\right) \alpha_1^{k}\alpha_2^{l}\right) = \psi\left(\alpha_1^{k+2}\alpha_2^{l}\right) + \psi\left(\alpha_1^{k}\alpha_2^{l+2}\right) -\psi\left(\alpha_1^{k}\alpha_2^{l}\right)
\end{align*}
vanishes. Clearly it does if either $k$ or $l$ is odd, and if $k = 2m$ and $l = 2n$, the right-hand side simplifies to $\Omega \left(m+1,n\right) + \Omega\left(m,n+1\right) - \Omega\left(m,n\right) = 0$ by the Pascal-like property.
\end{proof}

Finally, to prove the Invariance of $\psi$, we need the following three lemmas.
\begin{lemma}\label{Lemma degree of action}
    For any $h \in \SO$ and natural numbers $k, l$, each term of $h \cdot \alpha_1^k\alpha_2^l$ has degree $k+l$. 
\end{lemma}
\begin{proof}
    By using (\ref{G action}), $h \cdot \alpha_1^k \alpha_2^l = \left(h_{11} \alpha_1 + h_{21} \alpha_2\right)^k \left(h_{12} \alpha_1 + h_{22} \alpha_2\right)^l$. Expanding this using the Binomial Theorem, we see that the degree of each term is always $k+l$.
\end{proof}

\begin{lemma}
\label{Lemma pre-Invariance} For any natural numbers $m$ and $h \in \SO$, we have that $\psi\left(h \cdot \alpha_1^{2m}\right) = \psi\left(\alpha_1^{2m}\right)$ and $\psi\left(h \cdot \alpha_1^{2m-1}\alpha_2\right) = \psi\left(\alpha^{2m-1}\alpha_2\right)$.
\end{lemma}

\begin{proof}
The statement is obviously true for $h = -I$. Now for $h = h_u$ defined in Corollary \ref{Parametrization of SO},
\begin{align*}
\psi\left(h_{u} \cdot \alpha_1^{2m}\right) = \sum_{s=0}^{2m} \binom{2m}{s} \left(\frac{1-u^{2}}{1+u^{2}}\right)^{s} \left(\frac{2u}{1+u^{2}}\right)^{2m-s} \psi\left(\alpha_1^{s} \alpha_2^{2m-s}\right)
\end{align*}
but since the odd indices do not contribute to the sum, we just need to consider the even indices:
\begingroup
\allowdisplaybreaks
\begin{align}
\psi\left(h_{u} \cdot \alpha_1^{2m}\right) &= \sum_{s=0}^{m} \binom{2m}{2s} \left(\frac{1-u^{2}}{1+u^{2}}\right) ^{2s} \left(\frac{2u}{1+u^{2}}\right)^{2m-2s} \psi\left(\alpha_1^{2s}\alpha_2^{2m-2s}\right) \nonumber \\
&= \sum_{s=0}^{m} \binom{2m}{2s} \Omega\left(s,m-s\right) \left(\frac{1-u^{2}}{1+u^{2}}\right)^{2s} \left( \frac{2u}{1+u^{2}}\right)^{2m-2s} 1_{\F} \nonumber \\
& =\frac{\left(2m\right)!}{4^{m}m!m!} \sum_{s=0}^{m} \binom{m}{s} \left(
\frac{1-u^{2}}{1+u^{2}}\right)^{2s}\left(\frac{2u}{1+u^{2}}\right)^{2m-2s} 1_{\F}. \label{Pre-invariance}
\end{align}
\endgroup
Using the Binomial Theorem, (\ref{Pre-invariance}) simplifies to
\begin{equation*}
\psi\left(\alpha_1^{2m}\right) \left(\left(\frac{1-u^{2}}{1+u^{2}}\right)^{2} + \left(\frac{2u}{1+u^{2}}\right) ^{2}\right)^{m} = \psi\left(\alpha_1^{2m}\right)1_{\F}.
\end{equation*}
The proof that $\psi\left(h_{u} \cdot \alpha_1^{2m-1}\alpha_2\right) = \psi\left(\alpha_1^{2m-1}\alpha_2\right)$ is more involved but done similarly.
\end{proof}

\begin{lemma}\label{Locality and Invariance action}
    If $\pi \in \Pol$ evaluates to the zero function on $S^1$, then so does $h \cdot \pi$ for any $h \in \SO$.
\end{lemma}
\begin{proof}
    For an arbitrary $h \in \SO$, the map $\left[x_1, x_2\right] \mapsto \left[x_1,x_2\right] \cdot h$ is a bijection on $S^1$. Choose any $\left[x_1, x_2\right] \in S^1$, then $\left[x_1,x_2\right] = \left[u_1,u_2\right] \cdot h^{-1}$ for some $\left[u_1,u_2\right] \in S^1$. Using (\ref{SO action respects evaluation}),
    \begin{align*}
        \varepsilon(h \cdot \pi)(x_1,x_2) = \left(h \cdot \varepsilon(\pi)\right)(x_1,x_2) = \varepsilon(\pi)(\left[x_1,x_2\right] \cdot h) = \varepsilon(\pi)(u_1,u_2) = 0
    \end{align*}
    where the last equality follows from the assumption on $\pi$. Consequently $h \cdot \pi$ evaluates to the zero function on $S^1$.
\end{proof}

\begin{theorem}[Invariance of $\protect\psi$]
The linear functional $\psi$ satisfies the Invariance condition.
\end{theorem}

\begin{proof}
It is sufficient to show that $\psi\left(h \cdot \alpha_1^{k}\alpha_2^{l}\right)
= \psi\left(\alpha_1^{k}\alpha_2^{l}\right)$ for any $k,l \in \N$ and $h \in \SO$. As before, the statement is obviously true for $h=-I$, so we will only show that $\psi\left(h_{u} \cdot \alpha_1^{k}\alpha_2^{l}\right) = \psi\left(\alpha_1^{k}\alpha_2^{l}\right)$. If $k+l$ is odd, then by Lemma \ref{Lemma degree of action} each term of $h_u \cdot \alpha_`^k \alpha_2^l$ has an odd degree and therefore $\psi\left(h_{u} \cdot \alpha_1^{k} \alpha_2^{l}\right) = 0 = \psi\left(\alpha_1^{k}\alpha_2^{l}\right)$.

The polynomial $\pi = \alpha_1^{2m}\alpha_2^{2n} - \alpha_1^{2m} \left(1-\alpha_1^{2}\right)^{n}$ evaluates to the zero function on $S^1$ and therefore by Lemma \ref{Locality and Invariance action}, 
\begin{eqnarray*}
\psi\left( h_{u} \cdot \alpha_1^{2m}\alpha_2^{2n}\right) &=& \psi\left(
h_{u}\cdot \alpha_1^{2m} \left(1-\alpha_1^{2}\right)^{n}\right) = \sum_{s=0}^{n} \left(-1\right)^{s} \binom{n}{s} \psi \left(h_{u} \cdot \alpha_1^{2m+2s}\right).
\end{eqnarray*}
Now by Lemma \ref{Lemma pre-Invariance}, $\psi\left(h_{u} \cdot \alpha_1^{2m+2s}\right) = \psi\left(\alpha_1^{2m+2s}\right)$ and therefore this lets us retrace the steps:
\begin{eqnarray*}
\psi\left(h_{u} \cdot \alpha_1^{2m} \alpha_2^{2n}\right) = \sum_{s=0}^{n} \left(-1\right)^{s} \binom{n}{s} \psi\left(\alpha_1^{2m+2s}\right) = \psi\left(\alpha_1^{2m} \left(1-\alpha_1^{2}\right)^{n}\right) = \psi\left(\alpha_1^{2m}\alpha_2^{2n}\right)
\end{eqnarray*}
where in the last equality we used the Locality of $\pi$ again. The case $k$ and $l$ are both odd is treated similarly. The conclusion thus follows by the linearity of $\psi$.
\end{proof}

Next, we proceed to show that $\psi$ is the only circular integral functional on $\Pol$ with respect to $S^1$.

\begin{theorem}[Existence implies uniqueness]
If $\phi$ is any circular integral functional on $\Pol$ with respect to $S^1$, then $\phi$ is uniquely determined.
\end{theorem}

\begin{proof}
By linearity, it suffices to show that $\phi(\alpha_1^{k} \alpha_2^{l})$ is uniquely determined for any $k,l\in \N$. Using the Invariance property with $h=-I$, we obtain $\phi(\alpha_1^{k}\alpha_2^{l}) =\phi(-I \cdot \alpha_1^{k}\alpha_2^{l}) = (-1)^{k+l}\phi(\alpha_1^{k} \alpha_2^{l})$ so $\phi(\alpha_1^{k}\alpha_2^{l}) =0$ whenever $k+l$ is odd.

For $m \ge 1$, another application of the Invariance property with $h = h_u$ from Corollary \ref{Parametrization of SO} gives
\begin{align*}
\phi\left(\alpha_1^{2m}\right) =\phi\left( h_{u} \cdot \alpha_1^{2m}\right)
= \frac{1}{\left( 1+u^{2}\right)^{2m}} \sum_{s=0}^{2m} \binom{2m}{s}\left(
1-u^{2}\right)^{s} \left(2u\right)^{2m-s} \phi\left(\alpha_1^{s} \alpha_2^{2m-s}\right).
\end{align*}
We multiply both sides by $\left( 1+u^{2}\right)^{2m}$ and split the summation depending on the parity of $s$ to obtain 
\begin{align}
\left(1+u^{2}\right)^{2m} \phi\left(\alpha_1^{2m}\right) &= \sum_{s=0}^{m} \binom{2m}{2s} \left(1-u^{2}\right)^{2s} \left(2u\right)^{2m-2s} \phi\left(\alpha_1^{2s} \alpha_2^{2m-2s}\right) + \notag \\
&\quad\, \sum_{s=1}^{m} \binom{2m}{2s-1}\left(1-u^{2}\right)^{2s-1} \left(2u\right)^{2m-2s+1} \phi\left(\alpha_1^{2s-1} \alpha_2^{2m-2s+1}\right).
\label{Invariance rotation blue 0}
\end{align}

The polynomials $\pi_{1} = \alpha_1^{2s} \alpha_2^{2m-2s} - \alpha_1^{2s} \left(1-\alpha_1^{2}\right)^{m-s}$ and $\pi_{2} = \alpha_1^{2s-1}\alpha_2^{2m-2s+1} - \alpha_1^{2s-1} \left(1-\alpha_1^{2}\right)^{m-s}\alpha_2$ both evaluate to the zero function on $S^1$ so by Locality, we must have that 
\begin{align}
\phi\left(\alpha_1^{2s}\alpha_2^{2m-2s}\right) &= \phi\left(\alpha_1^{2s}\left(1-\alpha_1^{2}\right)^{m-s}\right) = \sum_{t=0}^{m-s}\left(
-1\right)^{t} \binom{m-s}{t} \phi\left(\alpha_1^{2s+2t}\right),
\label{Locality applied to pi_1} \\
\phi\left(\alpha_1^{2s-1} \alpha_2^{2m-2s+1}\right) &= \phi\left(\alpha_1^{2s-1} \left(1 - \alpha_1^{2}\right)^{m-s}\alpha_2\right)
= \sum_{t=0}^{m-s} \left(-1\right)^{t} \binom{m-s}{t} \phi\left(\alpha_1^{2s+2t-1}\alpha_2\right)  \label{Locality applied to pi_2}
\end{align}
respectively. By (\ref{Locality applied to pi_1}) and (\ref{Locality applied to pi_2}), equation (\ref{Invariance rotation blue 0})  becomes
\begin{eqnarray*}
\left(1+u^{2}\right)^{2m}\phi\left(\alpha_1^{2m}\right) &=& \sum_{s=0}^{m} \sum_{t=0}^{m-s} \left(-1\right)^{t} \binom{2m}{2s} \binom{m-s}{t} \left(1-u^{2}\right)^{2s} \left(2u\right)^{2m-2s} \phi\left(\alpha_1^{2s+2t}\right) + \\
 &&\sum_{s=0}^{m} \sum_{t=0}^{m-s} \left(-1\right)^{t} \binom{2m}{2s-1} \binom{m-s}{t} \left(1-u^{2}\right) ^{2s-1} \left(2u\right)^{2m-2s+1} \phi\left(\alpha_1^{2s+2t-1}\alpha_2\right).
\end{eqnarray*}

Now the following polynomial of degree at most $4m$ in $\F\left[\beta\right]$, namely
\begin{align*}
\pi &= \left(1+\beta^{2}\right)^{2m} \phi\left(\alpha_1^{2m}\right) - \sum_{s=0}^{m}\sum_{t=0}^{m-s} \left(-1\right)^{t} \binom{2m}{2s} \binom{m-s}{t} \left(1-\beta ^{2}\right)^{2s} \left(2\beta\right)^{2m-2s} \phi\left(\alpha_1^{2s+2t}\right) - \\
&\quad\, \sum_{s=0}^{m}\sum_{t=0}^{m-s} \left(-1\right)^{t} \binom{2m}{2s-1} \binom{m-s}{t} \left(1-\beta^{2}\right)^{2s-1} \left(2\beta\right)^{2m-2s+1} \phi\left(\alpha_1^{2s+2t-1}\alpha_2\right)
\end{align*}
has infinitely many roots, so $\pi$ is identically zero. By extracting the coefficient of $\beta$ and $\beta^2$ respectively we get $4m\phi\left(\alpha_1^{2m-1}\alpha_2\right) = 0$ and $8m^{2}\phi\left(\alpha_1^{2m}\right) -4m \left(2m-1\right) \phi\left(\alpha_1^{2m-2}\right) = 0$.

Since $m$ is arbitrary, we must have for any $m \ge 1$, $\phi\left(\alpha_1^{2m-1} \alpha_2\right) =0$ and the first-order recurrence relation $2m \phi\left(\alpha_1^{2m}\right) = \left(2m-1\right) \phi\left(\alpha_1^{2m-2}\right)$ with the initial condition $\phi\left(\mathbf{1}\right) = 1_{\mathbb{F}}$. Thus we see that $\phi\left(\alpha_1^{2m}\right)$ and $\phi\left(\alpha_1^{2m-1}\alpha_2\right)$ are uniquely determined for all $m \ge 1$.

Finally, by utilizing the Locality condition again, both $\phi\left(\alpha_1^{2m+1}\alpha_2^{2n+1}\right)$ and $\phi\left(\alpha_1^{2m}\alpha_2^{2n}\right)$ are uniquely determined since
\begin{align*}
\phi\left(\alpha_1^{2m+1}\alpha_2^{2n+1}\right) &= \phi\left(\alpha_1^{2m+1}\left(1-\alpha_1^{2}\right)^{n}\alpha_2 \right) = \sum_{s=0}^{n}\left(-1\right)^{n} \binom{n}{s} \phi\left(\alpha_1^{2m+2s+1}\alpha_2\right), \\
\phi\left(\alpha_1^{2m}\alpha_2^{2n}\right) &= \phi\left(\alpha_1^{2m} \left(1-\alpha_1^{2}\right)^{n} \right) = \sum_{s=0}^{n} \left(-1\right)^{n} \binom{n}{s} \phi\left(\alpha_1^{2m+2s}\right).
\end{align*}
This concludes the proof.
\end{proof}

\section{Generalization to Arbitrary Circles}

Fix a point $\left[a,b\right] \in \A$ and a non-zero $r\in \F$. We define $S^1_{r,\left[a,b\right]}$ to be the
collection of points $\left[x,y\right] \in \A$ such that $\left(x-a\right)^2 + \left(y-b\right)^2 = r^2$. A linear functional $\phi _{r,\left[a,b\right]} \colon \Pol \rightarrow \F$ is a circular integral functional on $\Pol$ with respect to $S^1_{r,\left[a,b\right]}$ if the following conditions are satisfied:

\begin{description}
    \item[(Normalization)] For the multiplicative identity $\mathbf{1}$ of $\Pol$, we have $\phi _{r,\left[a,b\right]}(\mathbf{1}) = r$.

    \item[(Locality)] If $\pi \in \Pol$ such that $\varepsilon(\pi) = 0$ on $S^1_{r,\left[a,b\right]}$, we have $\phi _{r,\left[a,b\right]} \left(\pi\right) = 0$.

    \item[(Invariance)] For any $\pi \in \Pol$ and $h \in \SO$, $\phi _{r,\left[a,b\right]} \left(h \cdot \pi \right) = \phi _{r,\left[a,b\right]}\left(\pi\right)$. 
\end{description}

By employing the same analysis, the existence and uniqueness of a circular integral functional on $\Pol$ with respect to $S^1_{r,\left[a,b\right]}$ can be derived from that of $\psi$.

\begin{theorem} \label{Existence and uniqueness of CIF general circle}
There is one and only one circular integral functional on $\Pol$ with respect to $S^1_{r,\left[a,b\right]}$, given by
\begin{align*}
\psi _{r,\left[a,b\right]}\left(\alpha_1^{k}\alpha_2^{l}\right) =  r\psi\left(\left(a+r\alpha_1\right)^k\left(b+r\alpha_2\right)^l\right)
\end{align*}
where $\psi$ is the circular integral functional on $\Pol$ with respect to $S^1$.
\end{theorem}

Now we are finally able to give an algebraic interpretation of the super Catalan numbers $S(m,n)$.

\begin{theorem}[An algebraic interpretation of $S(m,n)$]
Over $\Q$, for any $m, n \in \N$, we have that $2S(m,n) = \psi_{2,\left[0,0\right]}\left(\alpha_1^{2m} \alpha_2^{2n}\right)$.
\end{theorem}

\begin{proof}
It follows immediately from Theorem \ref{Existence and uniqueness of CIF general circle}. We have that
\begin{equation*}
    \psi_{2,\left[0,0\right]}\left(\alpha_1^{2m} \alpha_2^{2n}\right) = 2\psi\left(\left(2\alpha_1\right)^{2m}\left(2\alpha_2\right)^{2n}\right) = \frac{2^{2m+2n+1}}{4^{m+n}} S\left(m,n\right)1_{\Q} = 2S\left(m,n\right)
\end{equation*}%
as desired.
\end{proof}

\section*{Acknowledgement}
The author thanks Hopein Tang for the helpful discussions.

\end{document}